\documentclass[10pt]{article}
\usepackage[T1]{fontenc}
\usepackage{amsmath,amsfonts,amsthm,mathrsfs,amssymb}
\usepackage{cite}
\usepackage{multirow}
\usepackage{graphicx,float}
\usepackage{subfigure}
\usepackage{placeins}
\usepackage{color}
\usepackage{indentfirst}
\usepackage[bookmarks=true]{hyperref}
\numberwithin{equation}{section}
\newcommand{\R}{{\mathbb R}}
\newcommand{\Z}{{\mathbb Z}}
\newcommand{\N}{{\mathbb N}}
\newcommand{\C}{{\mathbb C}}

\newcommand{\be}{\begin{eqnarray}}
\newcommand{\ben}{\begin{eqnarray*}}
\newcommand{\en}{\end{eqnarray}}
\newcommand{\enn}{\end{eqnarray*}}

\newcommand{\pa}{\partial}

\newcommand{\curl}{{\rm curl\,}}

\newcommand{\divv}{{\rm div\,}}

\newcommand{\G}{\Gamma}

\newtheorem{theorem}{Theorem}[section]
\newtheorem{lemma}[theorem]{Lemma}
\newtheorem{corollary}[theorem]{Corollary}
\newtheorem{definition}[theorem]{Definition}
\newtheorem{remark}[theorem]{Remark}

\newtheorem{assumption}{Assumption}[section]

\definecolor{rot}{rgb}{1.000,0.000,0.000}
\definecolor{rot1}{rgb}{0.000,0.000,0.000}
\begin{document}
\renewcommand{\theequation}{\arabic{section}.\arabic{equation}}
\begin{titlepage}
\title{On the generalized Calder\'on formulas for closed- and open-surface elastic scattering problems}

\author{Liwei Xu\thanks{School of Mathematical Sciences, University of Electronic Science and Technology of China, Chengdu, Sichuan 611731, China. Email: {\tt xul@uestc.edu.cn}}\;,
Tao Yin\thanks{LSEC, Institute of Computational Mathematics and Scientific/Engineering Computing, Academy of Mathematics and Systems Science, Chinese Academy of Sciences, Beijing 100190, China, and School of Mathematical Science, University of Chinese Academy of Sciences, Beijing 100049, China. Email:{\tt yintao@lsec.cc.ac.cn}}}
\end{titlepage}
\maketitle

\begin{abstract}
The Calder\'on formulas (i.e., the combination of single-layer and hyper-singular boundary integral operators) have been widely utilized in the process of constructing valid boundary integral equation systems which could possess highly favorable spectral properties. This work is devoted to studying the theoretical properties of elastodynamic Calder\'on formulas which provide us with a solid basis for the design of fast boundary integral equation methods solving elastic wave problems defined on a close-surface or an open-surface in two dimensions. For the closed-surface case, it is proved that the Calder\'on formula is a Fredholm operator of second-kind except for certain circumstances. Regarding to the open-surface case, we investigate weighted integral operators instead of the original integral operators which are resulted from dealing with edge singularities of potentials corresponding to the elastic scattering problems by open-surfaces, and show that the Calder\'on formula is a compact perturbation of a bounded and invertible operator. To complete the proof, we need to use the well-posedness result of the elastic scattering problem,  the analysis of the zero-frequency integral operators defined on the straight arc, the singularity decompositions of the kernels of integral operators, and a new representation formula of the hyper-singular operator. Moreover, it can be demonstrated that the accumulation point of the spectrum of the invertible operator is the same as that of the eigenvalues of the Calder\'on formula in the closed-surface case.
\end{abstract}

\section{Introduction}
\label{sec:1}

As one of the most fundamental numerical methods, the boundary integral equation (BIE) method~\cite{HW08} has been extensively developed for numerical solutions of partial differential equations problems with various structures including bounded closed-surface~\cite{BK01,BXY17,BY20}, open screens~\cite{BL12,BL13,BXY21,LB15,WS90}, period or non-period infinite surface~\cite{BY21}, and so on. The BIE method has a feature of discretization of domains with lower dimensionality, and it is also a feasible method for the numerics of high frequency scattering problems. For large-scale problems with high-frequencies or three-dimensional complicated geometries, such iterative algorithms~\cite{DL15}  as the Krylov-subspace linear algebra solver GMRES, together with adequate acceleration techniques~\cite{BB21,BK01,L09},  are generally required for fast solving the resulting linear system whose coefficient matrix is dense. The efficiency of the GMRES iteration is highly related to the spectral features of the coefficient matrix of the linear system~\cite{TB97} and therefore, appropriate preconditioning, such as the analytical preconditioning based on the  Calder\'on formulas~\cite{BET12,CDL20} and the algebraic preconditioning strategies~\cite{BT98}, are usually  employed. However, only a few theoretical properties of the Calder\'on formulas (also called the Calder\'on relation in this work) for acoustic/electromagnetic closed-surface problems~\cite{BET12,CDL20}, two-dimensional acoustic open-surface problems~\cite{BL12,LB15} and elastic closed-surface problems with the standard traction operator~\cite{BXY21,BY21}, have been studied in  open literatures. We also refer to \cite{HJU18,HJU20,HU201,HU202} and the references therein for the study of inverses of integral operators on disks and the corresponding preconditioning associated with boundary element Galerkin discretizations.

This work is devoted to studying the theoretical properties of the Calder\'on formulas related to the two-dimensional problems of elastic scattering by closed- or open-surfaces which have many significant applications in science and engineering~\cite{PG18,W06}, including geophysics, non-destructive testing of solids materials, mining and energy production, etc. A fundamental purpose of utilizing the Calder\'on formulas is to construct BIEs, for example, the second-kind Fredholm integral equations,  with the highly favorable spectral properties that the eigenvalues of the BIEs are bounded away from zero and infinity. One can refer to the methodologies discussed in~\cite{BET12} for the acoustic case and those in \cite{CN02} for the electromagnetic case. Although for the acoustic and elastodynamic problems, the Calder\'on formulas are indeed the composition of the single-layer integral operator $S$ and the hyper-singular integral operator $N$, the extension of the theoretical analysis on the Calder\'on formulas in acoustics to that on the elastodynamic cases, however, encounters additional challenges. More precisely, for the smooth closed-surface case, the acoustic Calder\'on formula  reads $NS=-I/4+(D^*)^2$ where $D^*$ represents the transpose of the double-layer boundary integral operator, and is compact in appropriate Sobolev spaces. This fact ensures that the acoustic Calder\'on formula is of the second-kind Fredholm type. However, the corresponding operator $D^*$ in the elastic case is not compact, see for example~\cite{AJKKY,AKM}. In addition, the highly singular character of the associated integral kernel in elastodynamic hyper-singular operator is much more complicated than that in the acoustic case.

For the closed-surface case, by applying the polynomial compactness of the statistic elastic Neumann-Poincar\'e double-layer operator $D_0$ and its transpose $D_0^*$, it has been shown in~\cite{BXY21,BY20} that the elastodynamic Calder\'on formula involving the standard traction operator (\ref{stress}) is exactly a second-kind Fredholm operator (see~\cite{ZXY21} for two dimensional poroelastic case) whose eigenvalues are bounded away from zero and infinity with accumulation points being dependent on the elastic Lam\'e parameters. In addition, on the basis of another special choice of the traction operator (see Lemma~\ref{specialcompact}(ii)), the corresponding elastodynamic Calder\'on formulas are Fredholm operators of second-kind in both two and three dimensions as well. In this paper, a generalized  traction operator~\cite{H98,DL15} related to the generalized Betti's formula~\cite{BHSY,KGBB79} will be considered,  and the general results to be presented in  Theorem~\ref{main1}  indicate that the generalized elastic Calder\'on formula is a Fredholm operator of second-kind except for the above two special forms of traction operator.

Unfortunately, the properties of the closed-surface Calder\'on formula become invalid in the open-surface case. As being verified in~\cite{LB15,PS60}, the two-dimensional acoustic composite operator $NS$, which is not a second-kind Fredholm operator, takes a local singularity like $1/d$ where $d$ denotes the distance between the node being considered and the nearby endpoint of the open-arc. The composite operator $NS$ in the elastic case suffers from analogous character and therefore, it can not be treated in order to obtain favorable features in the classical Sobolev spaces. Instead of discussing $NS$, in light of the singular character of the solutions of the single-layer and the hyper-singular BIEs for solving the corresponding  acoustic open-surface scattering problems~\cite{CDD03}, a novel acoustic version weighted Calder\'on formula $N^{\rm w}S^{\rm w}$ is proposed in~\cite{BL12,BL13}, and it can be written into a sum of an invertible operator and a compact operator~\cite{BL12,LB15}. According to what we have known, the similar theoretical analysis for elastic problems, including  acoustic and electromagnetic problems in three dimensions, still remains unavailable in existing literatures. As being numerically demonstrated in~\cite{BXY21,BY20} for elastic open-surface scattering problems, the elastic version $N^{\rm w}S^{\rm w}$ formula leads to a significant decreasing on the GMRES iterations compared to the un-preconditioned one for a given residual tolerance. As a significant complement to the above numerical observation, we present in the current work  a rigorous theory on the two-dimensional elastic Calder\'on formula $N^{\rm w}S^{\rm w}$ which actually can be viewed as a Fredholm integral operator of second kind and a compact perturbation of a bounded and invertible operator. In addition, the accumulation point of the spectrum of the invertible operator is the same as that of the eigenvalues of the elastic Calder\'on formula in the two-dimensional closed-surface case, see Remark~\ref{equivalence}.

The remainder of this paper is organized as follows. Section~\ref{sec:2} introduces the generalized  traction operator together with the elastic single-layer and hyper-singular boundary integral operators. Section~\ref{sec:3} investigates the spectral properties of the elastic Calder\'on formula in the closed-surface case and presents some regularized formulations of the elastic hyper-singular boundary integral operators. The elastic Calder\'on formula in the open-surface case is studied in Section~\ref{sec:4}: instead of the original elastic  boundary integral operators, weighted single-layer and hyper-singular boundary integral operators under certain edge singularity circumstance of potentials are introduced in Section~\ref{sec:4.1}; in terms of the analysis results of the elastic Calder\'on formula on a special straight open-arc and the singularity decompositions of the integration kernels, the spectral properties of the generalized elastic Calder\'on formula in the universal open-surface case is analyzed in Section~\ref{sec:4.2} and ~\ref{sec:4.3}. A conclusion is finally given in Section~\ref{sec:5}.

\section{Preliminaries}
\label{sec:2}

Let $\Gamma$ be a smooth closed-surface or open-surface in $\R^2$. Denote by   $\lambda,\mu$  $(\mu>0,\lambda+\mu>0)$ the Lam\'e parameters and let $\rho>0$ be the mass density of a linear isotropic and homogeneous elastic medium. Denote by $\omega>0$ the frequency. For elastic problems, the standard traction operator $T(\pa,\nu)$ on the boundary is defined as
\be
\label{stress}
T(\pa,\nu)u:=2 \mu \, \partial_{\nu} u + \lambda \,
\nu \, \divv u+\mu \tau \curl u,\quad u=(u_1,u_2)^\top,
\en
in which $\nu=(\nu^1,\nu^2){^\top}$ is the unit outer normal to the boundary $\G$, $\tau=(-\nu_2,\nu_1)^\top$ is the corresponding tangential vector, $\partial_\nu:=\nu\cdot\nabla$ denotes the normal derivative and $\curl u=\pa_2u_1-\pa_1u_2$. To produce the generalized elastic Caldr\'on relations, we consider a modified traction operator~\cite{H98} defined as follows:
\be
\label{astress}
\widetilde{T}(\pa,\nu)u:=(\mu+\widetilde{\mu}) \, \partial_{\nu} u +
\widetilde{\lambda} \, \nu \, \divv u+\widetilde{\mu}\tau \curl u,
\en
where $\widetilde{\lambda}+\widetilde{\mu}=\lambda+\mu$. Obviously, $\widetilde{T}=T$ holds if $\widetilde{\lambda}=\lambda, \widetilde{\mu}=\mu$.

In this work, we are interested in the theoretical properties of the Calder\'on formulas, i.e., the composite operator $N_\omega S_\omega$, for elastic closed-surface and open-surface scattering problems, where $S_\omega, N_\omega$ denote the elastodynamic single-layer and hyper-singular boundary integral operators, respectively, in the form of
\be
\label{SBIO}
S_\omega[\phi](x)  = \int_\G \Pi_\omega(x,y)\phi(y)ds_y,\quad x\in\G,
\en
and
\be
\label{HBIO}
N_\omega[\psi](x) = \widetilde{T}(\pa_x,\nu_x)\int_\G
(\widetilde{T}(\pa_y,\nu_y)\Pi_\omega(x,y))^\top\psi(y)ds_y,\quad x\in\G.
\en
Here, we denote by $\Pi_\omega(x,y)$ the fundamental displacement tensor of the time-harmonic Navier equation in $\R^2$. That is
\be
\label{navier}
\Delta^*\Pi_\omega(\cdot,y)+\rho\omega^2\Pi_\omega(\cdot,y)=-\delta_y(\cdot)\mathbb{I} \quad\mbox{in}\quad\R^2,
\en
where $\Delta^{*}$ denotes the Lam\'e operator given by
\ben
\label{LameOper}
\Delta^* = \mu\,\mbox{div}\,\mbox{grad} + (\lambda + \mu)\,\mbox{grad}\, \mbox{div}\,,
\enn
and $\mathbb{I}$ is the $2\times2$ identity matrix. It is known that $\Pi_\omega(x,y)$ admits the form~\cite{KGBB79}
\ben
\Pi_\omega(x,y)=\frac{1}{\mu}G_{k_s}(x,y)\mathbb{I}+\frac{1}{\rho\omega^2}
\nabla_x\nabla_x^\top
\left[G_{k_s}(x,y)-G_{k_p}(x,y)\right],
\enn
where $G_{k_j}(x,y),j=p,s$ represents the fundamental solution of the Helmholtz equation in $\R^2$:
\be
\label{HelmholtzFS}
G_{k_j}(x,y) = \frac{i}{4}H_0^{(1)}(k_j|x-y|), \quad x\ne y
\en
with wave numbers $k_j,j=p,s$, $H_0^{(1)}(\cdot)$ being the Hankel function of the first kind of order zero, and $i=\sqrt{-1}$. The wave numbers $k_s=\omega/c_p, k_p=\omega/c_s$ are called the wave number of the compressional and shear waves, respectively, where
\ben
c_p=\sqrt{\mu/\rho}\quad
\mbox{and}\quad c_s=\sqrt{(\lambda+2\mu)/\rho}.
\enn

\begin{remark}
The modified traction operator $\widetilde{T}$ can be formulated alternatively as
\be
\label{Tform2}
\widetilde{T}(\pa,\nu)u(x)= (\lambda+\mu)\nu(\nabla \cdot u) +
\mu\pa_{\nu}u + \widetilde{\mu} \mathcal{M}(\pa,\nu)u,
\en
where the G\"unter derivative operator $\mathcal{M}(\pa,\nu)$ is given by
\ben
\mathcal{M}(\pa,\nu)u= \pa_{\nu}u -\nu(\nabla\cdot u)+\tau \curl u=A\pa_\tau u,\quad A=\begin{pmatrix}
  0 & -1 \\
  1 & 0
\end{pmatrix}(=-A^\top).
\enn
This modified traction operator is equivalent to another generalized form discussed in~\cite{DL15}
\ben
T_\alpha=T-\alpha \mathcal{M}.
\enn
It follows that $\widetilde{T}=T_\alpha$ when $\widetilde{\mu}=\mu-\alpha$.
\end{remark}

\section{Calder\'on relation: closed-surface}
\label{sec:3}

Let $\Gamma$ be a smooth closed boundary. As shown in~\cite{DL15}, the Calder\'on identities
\be
\label{ClosedCaldron}
S_\omega N_\omega =-\frac{1}{4}I+(D_\omega)^2,\quad N_\omega S_\omega=-\frac{1}{4}I+(D_\omega^*)^2
\en
hold where the single-layer boundary integral operator $S_\omega$ and hyper-singular boundary integral operator $N_\omega$ are defined in (\ref{SBIO}) and (\ref{HBIO}), respectively, and the operators $D_\omega, D_\omega^*$ defined as
\ben
D_\omega[\phi](x) &=& \int_\G
(\widetilde{T}(\pa_y,\nu_y)\Pi_\omega(x,y))^\top\phi(y)ds_y,\quad x\in\G, \\
D_\omega^*[\phi](x) &=& \widetilde{T}(\pa_x,\nu_x)\int_\G
\Pi_\omega(x,y)\phi(y)ds_y,\quad x\in\G,
\enn
are called the double-layer boundary integral operator and transpose of double-layer boundary integral operator, respectively. It is known~\cite{HW08} that $D_\omega:H^{1/2}(\Gamma)^2\rightarrow H^{1/2}(\Gamma)^2$ and $D_\omega^*:H^{-1/2}(\Gamma)^2\rightarrow H^{-1/2}(\Gamma)^2$ are linear bounded operators.

\begin{lemma}
\label{specialcompact}
(i). If $\widetilde{\mu}=\mu$, then $P_2(D_\omega):H^{1/2}(\Gamma)^2\rightarrow H^{1/2}(\Gamma)^2$ and $P_2(D_\omega^*): H^{-1/2}(\Gamma)^2\rightarrow H^{-1/2}(\Gamma)^2$ are compact. Here, $P_2(t)=t^2-(C_{\lambda,\mu})^2I$, $I$ is the identity operator and $C_{\lambda,\mu}$ is a constant given by
\be
\label{constC}
C_{\lambda,\mu}=\frac{\mu}{2(\lambda+2\mu)}<\frac{1}{2}.
\en
(ii). If $\widetilde{\mu}=\frac{\mu(\lambda+\mu)}{\lambda+3\mu}$, then the operators $D_\omega, D_\omega^*$ themselves are compact.
\end{lemma}
\begin{proof}
The conclusion (i) was proved in~\cite[Theorem 1]{BXY21}. For (ii), we refer to~\cite{H98}.
\end{proof}

The property of the integral operators $D_\omega, D_\omega^*$ for general $\widetilde{\mu}$ are given in the following lemma.
\begin{lemma}
\label{generalcompact}
For all $\mu\in\R$, $\widetilde{P}_2(D_\omega)$ and $\widetilde{P}_2(D_\omega^*)$ are compact. Here,  $\widetilde{P}_2(t)=t^2-(\widetilde{C}_{\lambda,\mu,\widetilde{\mu}})^2I$ and $\widetilde{C}_{\lambda,\mu,\widetilde{\mu}}$ is a constant that depends on the Lam\'e parameters:
\be
\label{constC1}
\widetilde{C}_{\lambda,\mu,\widetilde{\mu}}= \frac{2\mu\widetilde{\mu}+(\lambda+\mu)(\widetilde{\mu}-\mu)}{4\mu(\lambda+2\mu)}.
\en
\end{lemma}
\begin{proof}
Let $D_0$ be the corresponding elastic double-layer operator in the zero-frequency case $\omega=0$ which is given by~\cite{HW08}
\ben
D_0[\phi](x) = \int_\G
(\widetilde{T}(\pa_y,\nu_y)\Pi_0(x,y))^\top\phi(y)ds_y,\quad x\in\G,
\enn
where
\ben
\Pi_0(x,y)=-C_{\lambda,\mu}^{(1)}\ln|x-y|\mathbb{I}+C_{\lambda,\mu}^{(2)}
\frac{(x-y)(x-y)^\top}{|x-y|^2}
\enn
with
\ben
C_{\lambda,\mu}^{(1)}=\frac{\lambda+3\mu}{4\pi\mu(\lambda+2\mu)},\quad C_{\lambda,\mu}^{(2)}=\frac{\lambda+\mu}{4\pi\mu(\lambda+2\mu)}.
\enn
Denote by $[\cdot]_{ij}, i,j=1,2$ the elements of a $2\times 2$ matrix. Then direct calculation using (\ref{astress}) yields that
\be
\label{part1}
-\left[\widetilde{T}(\pa_y,\nu_y)\ln|x-y|\mathbb{I}\right]_{ij} &=& (\lambda+\mu)\frac{\nu_{y,i}(x_j-y_j)}{|x-y|^2} +\mu\delta_{ij}\frac{\nu_y^\top(x-y)}{|x-y|^2} \nonumber\\
&\quad& +\widetilde{\mu}\frac{\nu_{y,j}(x_i-y_i)-\nu_{y,i}(x_j-y_j)}{|x-y|^2},
\en
and
\be
\label{part2}
\left[\widetilde{T}(\pa_y,\nu_y)\frac{(x-y)(x-y)^\top}{|x-y|^2}\right]_{ij} &=& -\widetilde{\lambda}\frac{\nu_{y,i}(x_j-y_j)}{|x-y|^2} +2(\mu+\widetilde{\mu})\frac{\nu_y^\top(x-y)}{|x-y|^4}(x_i-y_i)(x_j-y_j) \nonumber\\
&\quad& -(\widetilde{\mu}+\mu) \frac{\nu_{y,i}(x_j-y_j)+\nu_{y,j}(x_i-y_i)}{|x-y|^2}\nonumber\\
&\quad& -\widetilde{\mu} \frac{\nu_y^\top(x-y)\delta_{ij}-\nu_{y,j}(x_i-y_i)}{|x-y|^2}
\en
Then $C_{\lambda,\mu}^{(1)}\times(\ref{part1}) +C_{\lambda,\mu}^{(2)}\times(\ref{part2})$ gives
\ben
(\widetilde{T}(\pa_y,\nu_y)\Pi_0(x,y))^\top= 2\widetilde{C}_{\lambda,\mu,\widetilde{\mu}} \mathcal{K}_1(x,y)+\mathcal{K}_2(x,y)~,
\enn
where
\ben
\mathcal{K}_1(x,y)= \frac{\nu_y(x-y)^\top-(x-y)\nu_y^\top}{2\pi|x-y|^2},
\enn
\ben
\mathcal{K}_2(x,y)= \frac{2\mu^2+(\lambda+\mu)(\widetilde{\mu}-\mu)}{4\pi\mu(\lambda+2\mu)} \frac{\nu_y^\top(x-y)}{|x-y|^2}\mathbb{I} +\frac{(\lambda+\mu)(\widetilde{\mu}+\mu)}{2\pi\mu(\lambda+2\mu)} \frac{\nu_y^\top(x-y)}{|x-y|^4}(x-y)(x-y)^\top.
\enn
Define the integral operators $K_j,j=1,2$ in the sense of Cauchy principle value as follows:
\ben
K_j[\phi](x)=\int_\Gamma \mathcal{K}_j(x,y)\phi(y)ds_y,\quad j=1,2.
\enn
Due to the property that $\nu_y^\top(x-y)=O(|x-y|^{2})$~\cite{CK83}, we conclude that the operator $K_2:H^{1/2}(\Gamma)^2\rightarrow H^{1/2}(\Gamma)^2$ is compact. For the operator $K_1$, it has been proved in~\cite[Proposition 3.1]{AJKKY} that
$K_1^2-\frac{1}{4}I:H^{1/2}(\Gamma)^2\rightarrow H^{1/2}(\Gamma)^2$ is compact. Then it follows that $\widetilde{P}_2(D_0):H^{1/2}(\Gamma)^2\rightarrow H^{1/2}(\Gamma)^2$ is compact. Then the compactness of $\widetilde{P}_2(D_\omega)$ results immediately from the fact that $D_\omega-D_0:H^{1/2}(\Gamma)^2\rightarrow H^{1/2}(\Gamma)^2$ is compact, since $D_\omega-D_0$ has a at-most weakly singular kernel, and
\ben
\widetilde{P}_2(D_\omega)=\widetilde{P}_2(D_0)+(D_\omega-D_0)D_\omega+D_0(D_\omega-D_0).
\enn
The compactness of $\widetilde{P}_2(D_\omega^*)$ can be proved analogously.
\end{proof}

\begin{remark}
If $\widetilde{\mu}=\mu$ or $\widetilde{\mu}=\frac{\mu(\lambda-\mu)}{\lambda+3\mu}$, then $|\widetilde{C}_{\lambda,\mu,\widetilde{\mu}}|=|C_{\lambda,\mu}|$ and therefore, $P_2(D_\omega)$ and $P_2(D_\omega^*)$ are compact which includes the result in Lemma~\ref{specialcompact}(i). If $\widetilde{\mu}=\frac{\mu(\lambda+\mu)}{\lambda+3\mu}$, then we know that $\widetilde{C}_{\lambda,\mu,\widetilde{\mu}}=0$. In this case, the compactness of $D_\omega$ (Lemma~\ref{specialcompact}(ii)) can be deduced since $K_2$ is compact.
\end{remark}

As stated in the following theorem, the generalized elastic Calder\'on relation for closed-surface problem is a direct corollary of Lemma~\ref{generalcompact}.  Note that under the special values (\ref{theorempara}) of $\widetilde{\mu}$, it holds that $|\widetilde{C}_{\lambda,\mu,\widetilde{\mu}}|=\frac{1}{2}$.

\begin{theorem}
\label{main1}
If
\be
\label{theorempara}
\widetilde{\mu}=-\mu\quad\mbox{or}\quad \widetilde{\mu}=\frac{\mu(3\lambda+5\mu)}{\lambda+3\mu},
\en
the compositions $S_\omega N_\omega:H^{1/2}(\Gamma)^2\rightarrow H^{1/2}(\Gamma)^2$ and $N_\omega S_\omega:H^{-1/2}(\Gamma)^2\rightarrow H^{-1/2}(\Gamma)^2$ are compact. Otherwise, the compositions $S_\omega N_\omega$ and $N_\omega S_\omega$ are Fredholm operators of second-kind both of which consists of a non-empty sequence of eigenvalues converging to $-\frac{1}{4}+(\widetilde{C}_{\lambda,\mu,\widetilde{\mu}})^2$ (Fig.~\ref{Fig1}).
\end{theorem}

\begin{figure}[htb]
\centering
\begin{tabular}{cc}
\includegraphics[scale=0.4]{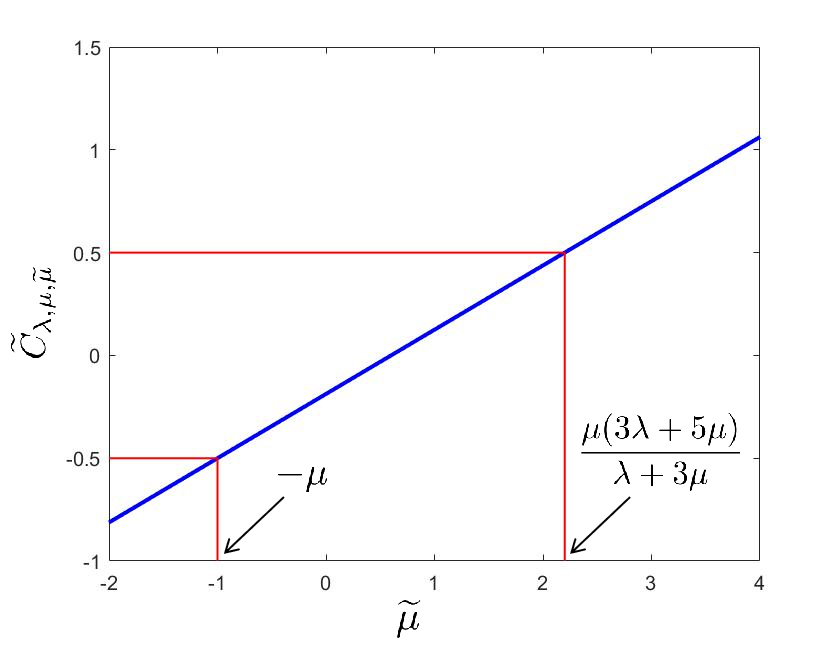} &
\includegraphics[scale=0.4]{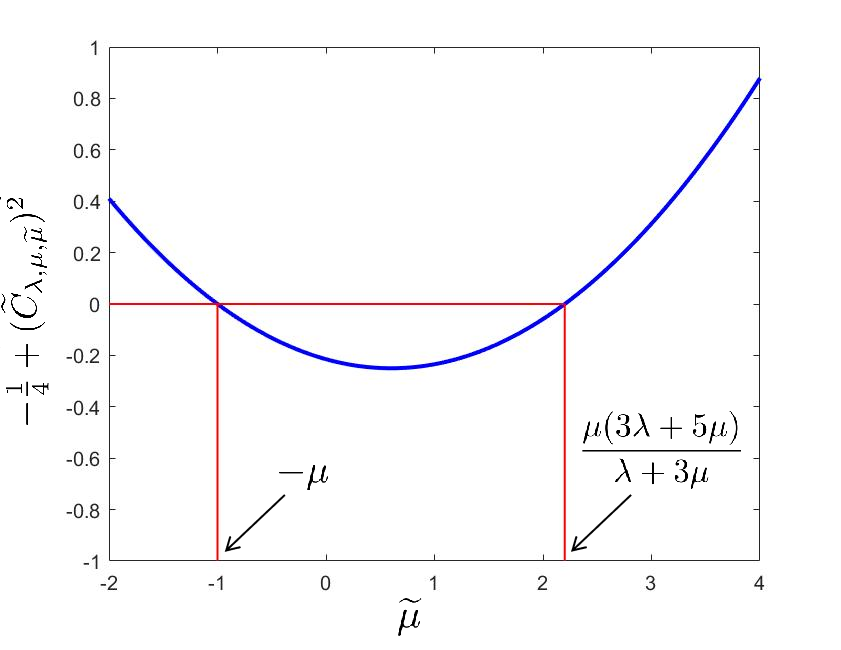}
\end{tabular}
\caption{Blue lines: values of $\widetilde{C}_{\lambda,\mu,\widetilde{\mu}}$ and $-\frac{1}{4}+(\widetilde{C}_{\lambda,\mu,\widetilde{\mu}})^2$ with respect to $\widetilde{\mu}$. Here, $\mu=1$, $\lambda=2$.}
\label{Fig1}
\end{figure}

Unfortunately, the generalized elastic Calder\'on relation for the closed-surface problem does not hold for the open-surface case any more, see Section~\ref{sec:4.1}. Before discussing the corresponding Calder\'on relation for the open-surface problem, we provide with some useful formulations in the closed-surface case which can be extended to the open-surface case for certain weighted integral operators introduced in Section~\ref{sec:4.1}. Let $N_0$ be the corresponding elastic hyper-singular operator in the zero-frequency case $\omega=0$ which is given by
\be
\label{N01}
N_0[\phi](x) = \widetilde{T}(\pa_x,\nu_x)\int_\G
(\widetilde{T}(\pa_y,\nu_y)\Pi_0(x,y))^\top\phi(y)ds_y,\quad x\in\G.
\en

\begin{lemma}
\label{lemmaN0}
The hyper-singular operator $N_0$ can be reformulated as
\be
\label{N02}
N_0[\phi](x) = \pa_{\tau_x}V_0[\pa_{\tau}\phi](x),
\en
where the integral operator $V_0:H^{-1/2}(\Gamma)^2\rightarrow H^{1/2}(\Gamma)^2$ is defined as
\ben
V_0[\chi](x) = \int_\Gamma \left[-C_{\lambda,\mu,\widetilde{\mu}}^{(1)} \ln|x-y|\mathbb{I}+C_{\lambda,\mu,\widetilde{\mu}}^{(2)} \frac{(x-y)(x-y)^\top}{|x-y|^2} \right] \chi(y)ds_y,
\enn
where
\ben
C_{\lambda,\mu,\widetilde{\mu}}^{(1)} &=& -\frac{(\widetilde{\mu}+\mu) [(\lambda+\mu)(\widetilde{\mu}-3\mu)+2\mu(\widetilde{\mu}-\mu)]}{4\pi\mu(\lambda+2\mu)},\\
C_{\lambda,\mu,\widetilde{\mu}}^{(2)} &=& \frac{(\lambda+\mu)(\widetilde{\mu}+\mu)^2}{4\pi\mu(\lambda+2\mu)}.
\enn
\end{lemma}
\begin{proof}
The proof of this lemma follows from the same steps as~\cite[Lemma 2.2.3]{HW08} and is omitted here.
\end{proof}
\begin{remark}
If $\widetilde{\mu}=\mu$, then $C_{\lambda,\mu,\widetilde{\mu}}^{(1)}=-C_{\lambda,\mu,\widetilde{\mu}}^{(2)}= \frac{\mu(\lambda+\mu)}{\pi(\lambda+2\mu)}$ and (\ref{N02}) is exactly the regularized formulation~\cite[(2.2.32)]{HW08}. If $\widetilde{\mu}=-\mu$, then $N_0=0$. If $\widetilde{\mu}=\frac{\mu(3\lambda+5\mu)}{\lambda+3\mu}$ which means that $C_{\lambda,\mu,\widetilde{\mu}}^{(1)}=0$, it follows that the kernel of $V_0$ is bounded!
\end{remark}

\begin{lemma}
The hyper-singular boundary integral operator $N_\omega$ can be expressed alternatively as
\be
\label{N2}
&\quad& N_\omega[\psi](x)\nonumber\\
&=& -\int_\Gamma\left[\rho\omega^2 (\nu_x\nu_y^\top-\nu_x^\top\nu_yI)G_{k_s}(x,y)-\widetilde{\mu}k_s^2G_{k_s}(x,y)J_{\nu_x,\nu_y}- \rho\omega^2G_{k_p}(x,y)\nu_x\nu_y^\top\right]\psi(y)ds_y\nonumber\\
&\quad& +(\mu+\widetilde{\mu})^2\pa_{\tau_x}\int_\Gamma
A\Pi_\omega(x,y)A\pa_{\tau_y}\psi(y) ds_y\nonumber\\
&\quad& +2(\mu+\widetilde{\mu})\pa_{\tau_x}\int_\Gamma G_{k_s}(x,y)\pa_{\tau_y}\psi(y)ds_y \nonumber\\
&\quad& -(\mu+\widetilde{\mu}) \int_\Gamma  \nu_{x} \nabla _x^\top [G_{k_s}(x,y)-G_{k_p}(x,y)]A\pa_{\tau_y}\psi(y)ds_y\nonumber\\
&\quad& -(\mu+\widetilde{\mu})\pa_{\tau_x}\int_\Gamma A\nabla_y
[G_{k_s}(x,y)-G_{k_p}(x,y)]\nu_y^\top \psi(y)ds_y.
\en
\end{lemma}
\begin{proof}
This result is a generalized form of~\cite[Theorem 6.2]{YHX17} and the proof is also omitted.
\end{proof}

\section{Calder\'on relation: open-surface}
\label{sec:4}

In this section, we study the elastic Calder\'on relation in the open-surface case. However, unlike the closed-surface case, an appropriate functional setting for the investigation of the composite operator $N_\omega S_\omega$ defined on an open-arc $\Gamma$  seems to be non-existing. As demonstrated in \cite[Appendix B]{LB15}, even for the simplest open-surface--straight arc $\{x_1\in[-1,1], x_2=0\}$, the acoustic composite operator $N_\omega S_\omega$ maps the constant function with value $1$ into a function possessing $d^{-1}$ edge singularity which does not belong to $H^{-1/2}(\Gamma)$, where $d=d(x)$ denotes the distance between $x$ and the corresponding end point for any $x$ in a neighbourhood of each end point. This difficulty also appears in the elastic open-surface case since both the acoustic and elastic integral operators $N_\omega, S_\omega$ defined on the straight arc contain similar singular kernels.

\subsection{Weighted integral operators}
\label{sec:4.1}

Instead of studying the operators $N_\omega, S_\omega$, we discuss the weighted forms $N_\omega^{\rm w}, S_\omega^{\rm w}$ resulting from certain regularity of potentials. Let $\Gamma$ be a smooth open arc in $\R^2$ and the unbounded domain $\R^2\backslash\Gamma$ is fulfilled with a linear isotropic and homogeneous elastic medium. Then the time-harmonic problem of elastic scattering by an open-surface can be modeled by the  Navier equation
\be
\label{navieropen}
\Delta^*u+\rho\omega^2u=0 \quad\mbox{in}\quad\R^2\backslash\Gamma,
\en
together with the boundary conditions
\be
\label{BC}
\begin{cases}
\mbox{Dirichlet}: & u=F \cr
\mbox{Neumann}: & \widetilde{T}(\pa,\nu)u=G
\end{cases}\quad\mbox{on}\quad \Gamma,
\en
and the Kupradze radiation condition at infinity~\cite{KGBB79}. Here, $u=(u_1,u_2)^\top$ denotes the displacement field. \begin{lemma}
The open-arc elastic scattering problem (\ref{navieropen})-(\ref{BC}) has at most one solution.
\end{lemma}
\begin{proof}
The proof of this result can be completed following the variational approach~\cite{WS90} together the generalized Rellich's lemma in elasticity (see~\cite[Lemma 2.14]{BHSY}).
\end{proof}

It is known that~\cite{SW84,WS90} the solutions of the Dirichlet and Neumann elastic problems of scattering by an open-arc admit the representations in forms of single- and double-layer potentials, respectively, i.e.,
\be
\label{DirichletS}
u(x)  = (\mathcal{S}_\omega\phi)(x):=  \int_{\Gamma}\Pi_\omega(x,y)) \phi(y)\,ds_y, \quad \forall\,x\in\R^2\backslash\Gamma,
\en
and
\be
\label{NeumannD}
u(x) = (\mathcal{D}_\omega\psi)(x):= \int_{\Gamma}(\widetilde{T}(\pa_y,\nu_y)\Pi_\omega(x,y))^\top \psi(y)\,ds_y, \quad \forall\,x\in\R^2\backslash\Gamma,
\en
respectively. Then the Dirichlet and Neumann problems reduce to the boundary integral
equations
\be
\label{BIEe}
S_\omega[\phi]=F,\quad N_\omega[\psi]=G \quad\mbox{on}\quad\G.
\en

\begin{definition}
An operator $L$ between two Sobolev spaces is called bicontinuous if it is continuous and invertible. As a corollary, the inverse $L^{-1}$ is also continuous. Assume that $\Gamma\subset\partial\Omega$ where $\partial\Omega$ is a smooth boundary of a bounded domain $\Omega$ in $\R^2$. We denote by $\widetilde{H}^s(\Gamma)$ the space of all $f\in H^s(\Gamma)$ satisfying ${\rm supp}(f)\subseteq\overline{\Gamma}$.
\end{definition}

\begin{lemma}
The operators $S_\omega: \widetilde{H}^{-1/2}(\Gamma)^2\rightarrow H^{1/2}(\Gamma)$ and $N_\omega: \widetilde{H}^{1/2}(\Gamma)^2\rightarrow H^{-1/2}(\Gamma)$ are bicontinuous under the assumption that
\be
\label{bicontinuous-condition}
\widetilde{\mu}\ne-\mu\quad\mbox{and}\quad \widetilde{\mu}\ne \frac{\mu(3\lambda+5\mu)}{\lambda+3\mu}.
\en
\end{lemma}
\begin{proof}
The bicontinuity of the operator $S_\omega$ holds analogously to \cite{SW84}. The assumption (\ref{bicontinuous-condition}) of $\widetilde{\mu}$ gives $C_{\lambda,\mu,\widetilde{\mu}}^{(1)}\ne 0$. By the integration kernel of $N_\omega$ together with the singularity decomposition (\ref{dec2}), it can be deduced that the coefficient of the weakly-singular part in the term
\ben
(\mu+\widetilde{\mu})^2A\Pi_\omega(x,y)A+2(\mu+\widetilde{\mu})G_{k_s}(x,y)\mathbb{I}
\enn
is non-zero. Following the proof of the solvability of the hyper-singular operator \cite[(3.6)]{WS90} results into the bicontinuity of the operator $N_\omega$.
\end{proof}

\begin{assumption}
\label{assume}
In the rest of this work, we always assume the value of $\widetilde{\mu}$ to be such that (\ref{bicontinuous-condition}) and additionally, $C_{\lambda,\mu,\widetilde{\mu}}^{(1)}\ln2+C_{\lambda,\mu,\widetilde{\mu}}^{(2)}\ne 0$ hold.
\end{assumption}

Denote by $d=d(x)$ a non-negative smooth function for $x\in\Gamma$ to represent the distance between  $x$ and the corresponding end point for any $x$ in a neighbourhood of each end point. Assuming that the right-hand sides $F, G$ in (\ref{BIEe}) are both infinitely differentiable, it is known~\cite{CDD03} that the density functions $\phi, \psi$ in (\ref{BIEe}) can be expressed in the forms
\be
\label{solsingular}
\phi=\frac{\phi^{\rm{w}}}{w},\quad \psi=w\;\psi^{\rm{w}},
\en
where $w=w(x)$ denotes a smooth function that reproduces the asymptotic $w\sim\sqrt{d}$ as $d\rightarrow 0$. It implies that $w/\sqrt{d}$ is infinitely differentiable up to the endpoints, and the new solutions $\phi^{\mbox{w}}, \psi^{\mbox{w}}$ are smooth up to the end points of $\Gamma$. Taking into account the solution singularities (\ref{solsingular}), we obtain the new boundary integral equations
\be
\label{BIEw}
S_\omega^{\rm{w}}[\phi^{\rm{w}}]=F,\quad N_\omega^{\rm{w}}[\psi^{\rm{w}}]=G \quad\mbox{on}\quad\G,
\en
where the weighted integral operators $S_\omega^{\rm{w}}, N_\omega^{\rm{w}}$ are defined as
\ben
S_\omega^{\rm{w}}[\phi^{\rm{w}}]=S_\omega\left[\frac{\phi^{\rm{w}}}{w}\right],\quad N_\omega^{\rm{w}}[\psi^{\rm{w}}]=N_\omega\left[w\;\psi^{\rm{w}}\right].
\enn
A similar regularized formulation of $N_\omega^{\rm{w}}$ can be obtained from (\ref{N2}) since the weight function $w$ in $N_\omega^{\mbox{w}}$ is smooth boundary-vanishing, see (\ref{Nce2}).

Without loss of generality, suppose that the boundary $\Gamma$ can be parameterized by means of a smooth vector function $x=x(t)=(x_1(t),x_2(t)), t\in[-1,1]$ satisfying $\mathcal{J}(t)=|x'(t)|\ne 0$. Here the prime $'$ denotes the derivative with respect to $t$. Choosing the smooth weighting function $w$ as $w(t)=\sqrt{1-t^2}$ yields
\ben
S_\omega^{\rm{w}}[\phi^{\rm{w}}](x(t))=\int_{-1}^1 \Pi_\omega(x(t),x(\iota)) \frac{\phi^{\rm{w}}(x(\iota))}{\sqrt{1-\iota^2}}\mathcal{J}(\iota)d\iota.
\enn
Then utilizing the changes of variables $t=\cos\theta, \iota=\cos\vartheta$ leads us to
\ben
\widetilde{S}_\omega^{\rm{w}}[\widetilde{\phi}^{\mbox{w}}](\theta)= F(x(\cos\theta)),\quad \widetilde{\phi}^{\rm{w}}(\theta)=\phi^{\mbox{w}}(x(\cos\theta)),
\enn
where the operator $\widetilde{S}_\omega^{\rm{w}}$ is given by
\ben
\widetilde{S}_\omega^{\rm{w}}[\widetilde{\phi}^{\rm{w}}](\theta) = \int_{0}^\pi \Pi_\omega(x(\cos\theta),x(\cos\vartheta)) \widetilde{\phi}^{\rm{w}}(\vartheta)\mathcal{J}(\cos\vartheta)d\vartheta.
\enn
The parameterized form $\widetilde{N}_\omega^{\rm{w}}$ corresponding to the integral operator $N_\omega^{\rm{w}}$ can be deduced in a similar manner.

Note that each smooth $\cos\theta$ dependence function can be extended to be a $2\pi$-periodic and even function. To study the properties of the parameterized operators $\widetilde{S}_\omega^{\rm{w}}, \widetilde{N}_\omega^{\rm{w}}$, we define the following Sobolev spaces:
\begin{definition}
For $s\in\R$, the Sobolev space $H^s_e(2\pi)$ is defined as the completion of space of infinitely differentiable $2\pi$-periodic and even functions defined in the real line with respect to the norm
    \ben
    \|v\|_{H^s_e(2\pi)}^2=|a_0|^2+2\sum_{m\in\Z^+} m^{2s}|a_m|^2,
    \enn
    where $a_m, m\in\Z^+$ denotes the coefficients in the cosine expansion of $v$:
    \ben
    v(\theta)=\frac{1}{2}a_0+\sum_{m\in\Z^+} a_m\cos(m\theta).
    \enn
\end{definition}

\subsection{Calder\'on relation: straight arc}
\label{sec:4.2}

Let $\Gamma$ be, specially, the straight arc $\{x_1\in[-1,1], x_2=0\}$ which means that $x_1(t)=t, x_2(t)=0, t\in[-1,1]$ for $x\in\Gamma$ and $|x'(t)|=1$. In this subsection, we consider the operators $\widetilde{S}_\omega^{\rm{w}}, \widetilde{N}_\omega^{\rm{w}}$ on the straight arc at zero frequency that are denoted by $\widetilde{S}_0, \widetilde{N}_0$, i.e.,
\be
\label{straightS0}
\widetilde{S}_0[\gamma](\theta) = \int_{0}^\pi\left[-C_{\lambda,\mu}^{(1)}\ln|\cos\theta-\cos\vartheta|\mathbb{I}+C_{\lambda,\mu}^{(2)}
\begin{pmatrix}
1 & 0\\
0 & 0
\end{pmatrix}\right] \gamma(\vartheta)d\vartheta,
\en
and
\be
\label{straightN0}
\widetilde{N}_0 = \widetilde{D}_0\widetilde{V}_0\widetilde{T}_0,
\en
where
\ben
\widetilde{D}_0[\gamma](\theta)=\frac{1}{\sin\theta}\frac{d\gamma(\theta)}{d\theta},\quad \widetilde{T}_0[\gamma](\theta)=\frac{d}{d\theta}(\gamma(\theta)\sin\theta),
\enn
\ben
\widetilde{V}_0[\gamma](\theta) = \int_{0}^\pi\left[-C_{\lambda,\mu,\widetilde{\mu}}^{(1)}\ln|\cos\theta-\cos\vartheta|\mathbb{I}+ C_{\lambda,\mu,\widetilde{\mu}}^{(2)}
\begin{pmatrix}
1 & 0\\
0 & 0
\end{pmatrix}\right] \gamma(\vartheta)d\vartheta.
\enn
The expression of $\widetilde{N}_0$ follows immediately from Lemma~\ref{lemmaN0} thanks to the smooth boundary-vanishing weight $w$.

For the basis $\{e_n=\cos n\theta:n\in\N\}$ of $H^s_e(2\pi), s\ge 0$, it can be derived from the diagonal property of Symm's operator~\cite{MH03} that
\be
\label{S0basis}
\widetilde{S}_0[e_n]=\begin{pmatrix}
\lambda_{1,n}^{S} & \\
& \lambda_{2,n}^{S}
\end{pmatrix}e_n,
\en
\ben
\lambda_{1,n}^{S}=\begin{cases}
\pi(C_{\lambda,\mu}^{(1)}\ln2+C_{\lambda,\mu}^{(2)}) , & n=0, \cr
\frac{\pi C_{\lambda,\mu}^{(1)}}{n}, & n\ge 1,
\end{cases}
\quad \lambda_{2,n}^{S}=\begin{cases}
\pi C_{\lambda,\mu}^{(1)}\ln2, & n=0, \cr
\lambda_{1,n}^{S}, & n\ge 1,
\end{cases}
\enn
and
\be
\label{V0basis}
\widetilde{V}_0[e_n]=\begin{pmatrix}
\lambda_{1,n}^{V} & \\
& \lambda_{2,n}^{V}
\end{pmatrix}e_n,
\en
\ben
\lambda_{1,n}^{V}=\begin{cases}
\pi(C_{\lambda,\mu,\widetilde{\mu}}^{(1)}\ln2+C_{\lambda,\mu,\widetilde{\mu}}^{(2)}) , & n=0, \cr
\frac{\pi C_{\lambda,\mu,\widetilde{\mu}}^{(1)}}{n}, & n\ge 1,
\end{cases}
\quad \lambda_{2,n}^{V}=\begin{cases}
\pi C_{\lambda,\mu,\widetilde{\mu}}^{(1)}\ln2, & n=0, \cr
\lambda_{1,n}^{V}, & n\ge 1.
\end{cases}
\enn
For $n=0$, we can obtain from (\ref{V0basis}) that
\ben
\widetilde{N}_0[e_0]= \widetilde{D}_0\widetilde{V}_0\widetilde{T}_0[e_0]= \widetilde{D}_0\widetilde{V}_0[e_1]=\widetilde{D}_0[\lambda_{1,1}^{V}\mathbb{I}e_1] =-\pi C_{\lambda,\mu,\widetilde{\mu}}^{(1)}\mathbb{I}.
\enn
For $n=1$,  it follows that
\ben
\widetilde{N}_0[e_1](\theta) =\widetilde{D}_0\widetilde{V}_0[e_2](\theta) =\widetilde{D}_0[\lambda_{1,2}^{V}\mathbb{I}e_2](\theta) =-2\pi C_{\lambda,\mu,\widetilde{\mu}}^{(1)}\mathbb{I}\cos\theta.
\enn
For $n\ge 2$, note that
\ben
\widetilde{T}_0[e_n](\theta)&=& \cos\theta\cos n\theta-n\sin n\theta\sin\theta\\
&=& \frac{e_{n+1}+e_{n-1}}{2}+ n\frac{e_{n+1}-e_{n-1}}{2}.
\enn
A direct application of (\ref{V0basis}) implies
\ben
\widetilde{V}_0\widetilde{T}_0[e_n]&=& \pi C_{\lambda,\mu,\widetilde{\mu}}^{(1)}\left(\frac{\frac{1}{n+1}e_{n+1}+\frac{1}{n-1}e_{n-1}}{2} +n\frac{\frac{1}{n+1}e_{n+1}-\frac{1}{n-1}e_{n-1}}{2}\right)\mathbb{I}\\
&=& \frac{\pi C_{\lambda,\mu,\widetilde{\mu}}^{(1)}}{2}(e_{n+1}-e_{n-1})
\enn
and then further for $n\ge 2$ we have
\be
\label{N0basis}
\widetilde{N}_0[e_n](\theta)= -\pi C_{\lambda,\mu,\widetilde{\mu}}^{(1)} \left(\frac{\cos\theta\sin n\theta}{\sin\theta} +n\cos n\theta\right)\mathbb{I},
\en
which also holds for $n=1$.

For the operators $\widetilde{D}_0,\widetilde{T}_0$, we have the following basic result, see \cite{LB15}.
\begin{lemma}
The operators $\widetilde{D}_0: H^2_e(2\pi)\rightarrow H^0_e(2\pi)$ and $\widetilde{T}_0: H^2_e(2\pi)\rightarrow H^1_e(2\pi)$ are bounded.
\end{lemma}

The properties of the operators $\widetilde{S}_0, \widetilde{N}_0, \widetilde{J}_0$ are presented in the following theorem.
\begin{theorem}
For all $s\ge0$, the operators $\widetilde{S}_0, \widetilde{V}_0: H^s_e(2\pi)^2\rightarrow H^{s+1}_e(2\pi)^2$, $\widetilde{N}_0: H^{s+1}_e(2\pi)^2\rightarrow H^s_e(2\pi)^2$ and $\widetilde{J}_0=\widetilde{N}_0\widetilde{S}_0: H^s_e(2\pi)^2\rightarrow H^s_e(2\pi)^2$ are all bicontinuous.
\end{theorem}
\begin{proof}
\begin{itemize}
\item Bicontinuity of $\widetilde{S}_0, \widetilde{V}_0$:
\end{itemize}

For $f=\sum_{n\in\N} f_ne_n\in H^s_e(2\pi)^2$, let the operators $W_S,W_V$ be defined as
\ben
W_S[f]=\sum_{n\in\N} \begin{pmatrix}
\lambda_{1,n}^{S} & \\
& \lambda_{2,n}^{S}
\end{pmatrix}f_ne_n,\quad W_V[f]=\sum_{n\in\N} \begin{pmatrix}
\lambda_{1,n}^{V} & \\
& \lambda_{2,n}^{V}
\end{pmatrix}f_ne_n.
\enn
Under the assumption (\ref{assume}), it follows that $W_S,W_V: H^s_e(2\pi)^2\rightarrow H^{s+1}_e(2\pi)^2$ are bicontinuous for all $s\ge 0$. In particular, $W_S,W_V$ are continuous from $H^0_e(2\pi)^2$ into $H^0_e(2\pi)^2$. For every basis element $e_n,n\in\N$, the operators $\widetilde{S}_0,\widetilde{V}_0$ coincide with $W_S,W_V$, respectively. Then $\widetilde{S}_0$ (resp. $\widetilde{V}_0$) and $W_S$ (resp. $W_V$) coincide on the dense set $\{e_n\}$ of $H^0_e(2\pi)^2$ and thus, coincide throughout $H^0_e(2\pi)^2$. Then the bicontinuity of $\widetilde{S}_0, \widetilde{V}_0: H^s_e(2\pi)^2\rightarrow H^{s+1}_e(2\pi)^2$ follows immediately.

\begin{itemize}
\item Bicontinuity of $\widetilde{J}_0$:
\end{itemize}

We first show that $\widetilde{J}_0: H^s_e(2\pi)^2\rightarrow H^s_e(2\pi)^2$ is continuous. Obviously, $\widetilde{J}_0: H^1_e(2\pi)^2\rightarrow H^0_e(2\pi)^2$ is continuous since $\widetilde{D}_0: H^2_e(2\pi)\rightarrow H^0_e(2\pi)$ \cite[Lemma 3.2]{LB15} and $\widetilde{T}_0: H^2_e(2\pi)\rightarrow H^1_e(2\pi)$ are continuous. Combining the relations (\ref{S0basis}) and (\ref{N0basis}) gives
\be
\label{Jbasis}
\widetilde{J}_0[e_n](\theta) =\begin{cases}
\begin{pmatrix}
\lambda_1^{J} & \\
& \lambda_2^{J}
\end{pmatrix}, & n=0, \cr
\lambda_3^{J} \left(\frac{\cos\theta\sin n\theta}{n\sin\theta} +\cos n\theta\right)\mathbb{I}, & n\ge1,
\end{cases}
\en
where $\lambda_1^{J}=-\pi^2 C_{\lambda,\mu,\widetilde{\mu}}^{(1)}(C_{\lambda,\mu}^{(1)}\ln2+C_{\lambda,\mu}^{(2)})$, $\lambda_2^{J}=-\pi^2 C_{\lambda,\mu,\widetilde{\mu}}^{(1)}C_{\lambda,\mu}^{(1)}\ln2$ and $\lambda_3^{J}=-\pi^2 C_{\lambda,\mu,\widetilde{\mu}}^{(1)}C_{\lambda,\mu}^{(1)}$ for $n\ge 1$. Let the operator $\widetilde{C}$ be given by
\be
\label{operatorC}
\widetilde{C}[\gamma](\theta)= \frac{\theta(\pi-\theta)}{\pi\sin\theta}\left[ \frac{1}{\theta}\int_0^\theta\gamma(s)ds -\frac{1}{\pi-\theta}\int_\theta^\pi\gamma(s)ds \right].
\en
It has been proved in ~\cite[Lemma 3.5]{LB15} that the operator $\widetilde{C}: H^s_e(2\pi)\rightarrow H^s_e(2\pi)$ is continuous and
\be
\label{Cbasis}
\widetilde{C}[e_n](\theta) =\begin{cases}
0, & n=0, \cr
\frac{\sin n\theta}{n\sin\theta}, & n\ge1.
\end{cases}
\en
Then we can rewrite $\widetilde{J}_0[e_n]$ as
\ben
\widetilde{J}_0[e_n]=\widetilde{W}_0[e_n],
\enn
where the integral operator $\widetilde{W}_0$ is defined as
\ben
\widetilde{W}_0[\gamma](\theta)= \lambda_3^{J}\mathbb{I}\gamma(\theta) +\lambda_3^{J}\cos\theta\mathbb{I}\widetilde{C}[\gamma](\theta) + \frac{1}{\pi}\begin{pmatrix}
\lambda_1^{J}-\lambda_3^{J} & \\
& \lambda_2^{J}-\lambda_3^{J}
\end{pmatrix}\int_0^\pi\gamma(s)ds.
\enn
The continuity of $\widetilde{W}_0: H^s_e(2\pi)\rightarrow H^s_e(2\pi)$ deduces from the continuity of $\widetilde{C}: H^s_e(2\pi)\rightarrow H^s_e(2\pi)$ for all $s\ge 0$ and therefore the continuity of $\widetilde{J}_0: H^s_e(2\pi)\rightarrow H^s_e(2\pi)$ results.

Next we prove the invertibility of $\widetilde{J}_0$ and its inverse $\widetilde{J}_0^{-1}: H^s_e(2\pi)^2\rightarrow H^s_e(2\pi)^2$ is given by
\be
\label{inverseJ}
\widetilde{J}_0^{-1}=-\left(\pi^2 (C_{\lambda,\mu,\widetilde{\mu}}^{(1)})^2\right)^{-1} \widetilde{S}_0^{-1}\widetilde{C}\widetilde{V}_0\widetilde{T}_0,
\en
for $s\ge 2$ and given by the unique continuous extension of the right-hand side of (\ref{inverseJ}) for $0\le s<2$. On one hand, it easily follows from~\cite[Corollary 3.13]{LB15} that $\widetilde{C}\widetilde{V}_0\widetilde{T}_0: H^{s}_e(2\pi)^2\rightarrow H^{s}_e(2\pi)^2, s\ge 1$ is continuous and it can be extended in a unique style to an operator which is continuous from $H^{s}_e(2\pi)^2$ to $H^{s+1}_e(2\pi)^2$ for all $s\ge 0$. For $n=0$, $\widetilde{D}_0[e_0]=0$. For $n\ge 1$,
\ben
\widetilde{D}_0[e_n]=-n\frac{\sin n\theta}{\sin\theta}.
\enn
Note that for $n\ge 1$,
\ben
\widetilde{V}_0[e_n]=\lambda_{1,n}^{V}\mathbb{I}[e_n],\quad \widetilde{C}[e_n]=\frac{\sin n\theta}{n\sin\theta}.
\enn
We conclude from the density of the basis $\{e_n\}$ in $H^s_e(2\pi)$ that
\ben
\widetilde{D}_0\mathbb{I}= -\pi^2 (C_{\lambda,\mu,\widetilde{\mu}}^{(1)})^2 \widetilde{C}(\widetilde{V}_0^{-1})^2.
\enn
This means that $\widetilde{D}_0: H^{s+2}_e(2\pi)\rightarrow H^{s}_e(2\pi), s\ge 0$ is continuous. Denote
\ben
\widetilde{I}_0=-\left(\pi^2 (C_{\lambda,\mu,\widetilde{\mu}}^{(1)})^2\right)^{-1}\widetilde{S}_0^{-1}\widetilde{C}\widetilde{V}_0\widetilde{T}_0.
\enn
Then
\ben
\widetilde{J}_0\widetilde{I}_0=\widetilde{C}\widetilde{V}_0^{-1}\widetilde{V}_0^{-1} \widetilde{V}_0\widetilde{T}_0\widetilde{S}_0 \widetilde{S}_0^{-1}\widetilde{C}\widetilde{V}_0\widetilde{T}_0 =\widetilde{C}\widetilde{V}_0^{-1} \widetilde{T}_0\widetilde{C}\widetilde{V}_0\widetilde{T}_0
\enn
with regularity
\ben
H^{s}_e(2\pi)^2\xrightarrow{\widetilde{C}\widetilde{V}_0\widetilde{T}_0} H^{s+1}_e(2\pi)^2 \xrightarrow{\widetilde{T}_0} H^{s}_e(2\pi)^2 \xrightarrow{\widetilde{V}_0^{-1}} H^{s-1}_e(2\pi)^2 \xrightarrow{\widetilde{C}} H^{s-1}_e(2\pi)^2, \quad s\ge 1.
\enn
Noting that $\widetilde{T}_0[e_0]=e_1$, $\widetilde{C}[e_1]=e_0$ and $\widetilde{T}_0\widetilde{C}[e_n]=e_n$ for all $n\ge 1$, we obtain $\widetilde{J}_0\widetilde{I}_0=I$ which means that $\widetilde{I}_0$ is a right inverse of $\widetilde{J}_0$ for $s\ge 1$. On the other hand, for $s\ge 2$,
\ben
\widetilde{I}_0\widetilde{J}_0=\widetilde{S}_0^{-1}\widetilde{C}\widetilde{V}_0\widetilde{T}_0 \widetilde{C}\widetilde{V}_0^{-1}\widetilde{V}_0^{-1} \widetilde{V}_0\widetilde{T}_0\widetilde{S}_0 =\widetilde{S}_0^{-1}\widetilde{C}\widetilde{V}_0\widetilde{T}_0 \widetilde{C}\widetilde{V}_0^{-1}\widetilde{T}_0\widetilde{S}_0,
\enn
Using the same argument for right inverse demonstration, it can be obtained that
\ben
\widetilde{I}_0\widetilde{J}_0=\widetilde{S}_0^{-1}\widetilde{C}\widetilde{T}_0\widetilde{S}_0.
\enn
Definition of the operators $\widetilde{C}, \widetilde{T}_0$ yields $\widetilde{C}\widetilde{T}_0[f]=f$ for all $f\in H^{s}_e(2\pi)$. Thus, $\widetilde{I}_0\widetilde{J}_0=I$ which means that $\widetilde{I}_0$ is a left inverse of $\widetilde{J}_0$ for $s\ge 2$. The continuity of $\widetilde{I}_0: H^s_e(2\pi)^2\rightarrow H^s_e(2\pi)^2$ for $s\ge 2$, which can be extended in a unique style as a continuous operator for $s\ge 0$, follows immediately from the continuity of $\widetilde{S}_0$ and $\widetilde{C}\widetilde{V}_0\widetilde{T}_0$. Therefore, the relations $\widetilde{J}_0\widetilde{I}_0=I$ and $\widetilde{I}_0\widetilde{J}_0=I$ can be extended to the case of $s\ge 0$ due to the density of $H^2_e(2\pi)$ in $H^s_e(2\pi)$, $0\le s<2$. This completes the proof of the bicontinuity of $\widetilde{J}_0$.

\begin{itemize}
\item Bicontinuity of $\widetilde{N}_0$:
\end{itemize}

The bicontinuity of $\widetilde{N}_0: H^{s+1}_e(2\pi)^2\rightarrow H^s_e(2\pi)^2$ results from the bicontinuity of $\widetilde{S}_0,\widetilde{J}_0$ and $\widetilde{N}_0=\widetilde{J}_0\widetilde{S}_0^{-1}$.
\end{proof}

\begin{theorem}
\label{spectrumJ0}
For all $s>0$, the point spectrum $\sigma_s$ of $\widetilde{J}_0$ can be expressed as the union
\ben
\sigma_s=\Lambda_s\cup\Lambda_\infty,
\enn
where $\Lambda_s$ is the open bounded set
\ben
\Lambda_s=\left\{\leftthreetimes=-\lambda_3^{J}(\lambda_x+i\lambda_y)\in\C: s+\frac{1}{2}<\frac{-(\lambda_x+1)}{(\lambda_x+1)^2+\lambda_y^2}, \lambda_x+1< 0\right\},
\enn
and $\Lambda_\infty$ is the discrete set
\ben
\Lambda_\infty=\left\{\lambda_1^{J}, \lambda_2^{J}, \lambda_3^{J}+\frac{\lambda_3^{J}}{n}:n\in\N\right\}.
\enn
Moreover, $\sigma_s$ is bounded away from zero and infinity.
\end{theorem}
\begin{proof}
Rewriting (\ref{Jbasis}) gives
\be
\label{Jbasis1}
\widetilde{J}_0[e_n] =\begin{cases}
\begin{pmatrix}
\lambda_1^{J} & \\
& \lambda_2^{J}
\end{pmatrix}, & n=0, \cr
-4\lambda_3^{J} \widetilde{j}_0[e_n]\mathbb{I}, & n\ge1,
\end{cases}
\en
where
\ben
\widetilde{j}_0[e_n](\theta)=-\frac{\sin(n+1)\theta}{4n\sin\theta}+\frac{\cos n\theta}{4n} -\frac{\cos n\theta}{4},\quad n\ge 1.
\enn
Then the point spectrum of $\widetilde{J}_0$ results from \cite[Lemma 3.16]{LB15}. Utilizing the polar coordinates $(r,\theta)$ around $(-1,0)$ the set $\Lambda_s$ can be rewritten as
\ben
\Lambda_s=\left\{\leftthreetimes=-\lambda_3^{J}(\lambda_x+i\lambda_y)\in\C: \lambda_x+1=r\cos\theta, \lambda_y=r\sin\theta, 0<r<-\frac{\cos\theta}{s+\frac{1}{2}}, \frac{\pi}{2}<\theta<\frac{3\pi}{2}\right\},
\enn
satisfying $\cap_{s>0}\lambda_s=\emptyset$ and $\cap_{s>0}\overline{\lambda_s}=\{\lambda_3^{J}\}$. Additionally, $\mbox{dist}(\sigma_s,0)=|\lambda_3^{J}|$ and
\ben
\max_{\lambda\in\sigma_s}|\leftthreetimes|=\max\{|\lambda_1^{J}|, |\lambda_2^{J}|, 3|\lambda_3^{J}|\}.
\enn
\end{proof}

\begin{remark}
\label{equivalence}
We point out in this remark that $\lambda_3^{J}=-\frac{1}{4}+(\widetilde{C}_{\lambda,\mu,\widetilde{\mu}})^2$ implying that the clustered point of the point spectrum of the Calder\'on formula $\widetilde{J}_0=\widetilde{N}_0\widetilde{S}_0$ (in the case of straight arc) is equivalent to the accumulation point of the eigenvalues of the Calder\'on formula in the closed-surface case (Theorem~\ref{main1}). In fact, on one hand,
\ben
\lambda_3^{J} &=& -\pi^2 C_{\lambda,\mu,\widetilde{\mu}}^{(1)}C_{\lambda,\mu}^{(1)}\\
&=& \frac{(\widetilde{\mu}+\mu) [(\lambda+\mu)(\widetilde{\mu}-3\mu)+2\mu(\widetilde{\mu}-\mu)]}{4\mu(\lambda+2\mu)} \frac{\lambda+3\mu}{4\mu(\lambda+2\mu)}.
\enn
On the other hand,
\ben
-\frac{1}{4}+(\widetilde{C}_{\lambda,\mu,\widetilde{\mu}})^2 = \left(\widetilde{C}_{\lambda,\mu,\widetilde{\mu}}-\frac{1}{2}\right) \left(\widetilde{C}_{\lambda,\mu,\widetilde{\mu}}+\frac{1}{2}\right).
\enn
Then the following direct evaluations
\ben
\widetilde{C}_{\lambda,\mu,\widetilde{\mu}}-\frac{1}{2} = \frac{2\mu\widetilde{\mu}+(\lambda+\mu)(\widetilde{\mu}-\mu)-2\mu(\lambda+2\mu)}{4\mu(\lambda+2\mu)} = \frac{(\lambda+\mu)(\widetilde{\mu}-3\mu)+2\mu(\widetilde{\mu}-\mu)}{4\mu(\lambda+2\mu)},
\enn
and
\ben
\widetilde{C}_{\lambda,\mu,\widetilde{\mu}}+\frac{1}{2} = \frac{2\mu\widetilde{\mu}+(\lambda+\mu)(\widetilde{\mu}-\mu)+2\mu(\lambda+2\mu)}{4\mu(\lambda+2\mu)} = \frac{(\widetilde{\mu}+\mu)(\lambda+3\mu)}{4\mu(\lambda+2\mu)},
\enn
indicate the argument.
\end{remark}

\subsection{Calder\'on relation: general open-surface}
\label{sec:4.3}

For a general smooth open-arc $\Gamma$, combining the Jacobian $\mathcal{J}(\theta)=|x'(\cos\theta)|$ and the operators $\widetilde{S}_0, \widetilde{N}_0$, we introduce the operators
\ben
\widetilde{S}_0^\mathcal{J}=\widetilde{S}_0\widetilde{Z}_0, \quad \widetilde{N}_0^\mathcal{J}=\widetilde{Z}_0^{-1}\widetilde{N}_0
\enn
and $\widetilde{J}_0^\mathcal{J}=\widetilde{N}_0^\mathcal{J}\widetilde{S}_0^\mathcal{J}$, where the operator $\widetilde{Z}_0$ is given by $\widetilde{Z}_0[\gamma](\theta)=\gamma(\theta)\mathcal{J}(\theta)$. Note that $\widetilde{J}_0^{\mathcal{J}}=\widetilde{Z}_0^{-1}\widetilde{J}_0\widetilde{Z}_0$, the following results hold.

\begin{corollary}
\label{corollaryJ}
For all $s\ge0$, the operators $\widetilde{S}_0^\mathcal{J}: H^s_e(2\pi)^2\rightarrow H^{s+1}_e(2\pi)^2$,  $\widetilde{N}_0^\mathcal{J}: H^{s+1}_e(2\pi)^2\rightarrow H^s_e(2\pi)^2$ and $\widetilde{J}_0^\mathcal{J}: H^s_e(2\pi)^2\rightarrow H^s_e(2\pi)^2$ are all bicontinuous. Moreover, the point spectrum of $\widetilde{J}_0^\mathcal{J}$ is equivalent to the point spectrum of $\widetilde{J}_0$ for $s>0$.
\end{corollary}

The following lemma   builds a relationship between the function spaces $H^s_e(2\pi),H^{s+1}_e(2\pi)$ and the original function spaces $\widetilde{H}^{-1/2}(\Gamma),\widetilde{H}^{1/2}(\Gamma)$, see \cite[Corollary 4.3, Lemma 4.9]{LB15}.
\begin{lemma}
\label{link}
For any $\widetilde{\phi}\in H^s_e(2\pi)$ and $\widetilde{\psi}\in H^{s+1}_e(2\pi)$ with $s>0$. Denote $\phi(t)=\widetilde{\phi}(\arccos(t)), \psi(t)=\widetilde{\psi}(\arccos(t)), t\in[-1,1]$. For all $x=x(t)\in\Gamma$, define $\alpha(x)=\phi(t)$, $\beta(x)=\psi(t)$ and $W(x)=w(t)$. Then we have $\alpha/W\in\widetilde{H}^{-1/2}(\Gamma)$ and $\beta W\in\widetilde{H}^{1/2}(\Gamma)$.
\end{lemma}

The generalized Calder\'on relations for elastic open-surface problem is summarized in the following theorem whose proof relies heavily on the singularity decompositions of the kernels of the integral operators as being stated in Appendix.
\begin{theorem}
\label{main}
For all $s>0$, the operators $\widetilde{S}_\omega^{\rm{w}}: H^s_e(2\pi)^2\rightarrow H^{s+1}_e(2\pi)^2$,  $\widetilde{N}_\omega^{\rm{w}}: H^{s+1}_e(2\pi)^2\rightarrow H^s_e(2\pi)^2$ and $\widetilde{J}_\omega^{\rm{w}}=\widetilde{N}_0\widetilde{S}_0: H^s_e(2\pi)^2\rightarrow H^s_e(2\pi)^2$ are all bicontinuous. Moreover, the following generalized Calder\'on relation
\ben
\widetilde{J}_\omega^{\rm{w}}=\widetilde{J}_0^\mathcal{J}+\widetilde{K}
\enn
holds, where $\widetilde{K}: H^s_e(2\pi)^2\rightarrow H^{s}_e(2\pi)^2$ is compact and $\widetilde{J}_0^\mathcal{J}: H^s_e(2\pi)^2\rightarrow H^{s}_e(2\pi)^2$ is bicontinuous and its point spectrum is given by $\sigma_s$, see Theorem~\ref{spectrumJ0}.
\end{theorem}
\begin{proof}
\begin{itemize}
\item Bicontinuity of $\widetilde{S}_\omega^{\rm{w}}$ and $\widetilde{V}_\omega^{\rm{w}}$:
\end{itemize}

For integer $s\ge 0$, in view of the singularity decomposition (\ref{dec1}) of the fundamental solution $\Pi_\omega$ we know that
\ben
\widetilde{S}_\omega^{\rm{w}}[\gamma](\theta) &=& \widetilde{S}_0^\mathcal{J}[\gamma](\theta)  +\sum_{n=2}^{s+3} A_n(\cos\theta)\int_0^\pi (\cos\theta-\cos\vartheta)^n \ln|\cos\theta-\cos\vartheta| \gamma(\vartheta)\mathcal{J}(\vartheta)d\vartheta\nonumber\\
&\quad& +\int_0^\pi \hat{A}_{s+3}(\cos\theta,\cos\vartheta) \gamma(\vartheta)\mathcal{J}(\vartheta)d\vartheta.
\enn
It follows from~\cite[Lemma 4.1]{LB15} that $\widetilde{S}_\omega^{\rm{w}}-\widetilde{S}_0^\mathcal{J}$ maps continuously from $H^s_e(2\pi)^2$ into $H^{s+3}_e(2\pi)^2$ and this result can be extended to the case of all $s\ge0$ by interpolation. Therefore, $\widetilde{S}_\omega^{\rm{w}}: H^s_e(2\pi)^2\rightarrow H^{s+1}_e(2\pi)^2$ is continuous. In view of the invertibility of $\widetilde{S}_0^\mathcal{J}$, the operator $\widetilde{S}_\omega^{\rm{w}}$ can be expressed as
\ben
\widetilde{S}_\omega^{\rm{w}}=\widetilde{S}_0^\mathcal{J} \left(I+(\widetilde{S}_0^\mathcal{J})^{-1} (\widetilde{S}_\omega^{\rm{w}}-\widetilde{S}_0^\mathcal{J})\right),
\enn
in which $(\widetilde{S}_0^\mathcal{J})^{-1} (\widetilde{S}_\omega^{\rm{w}}-\widetilde{S}_0^\mathcal{J}): H^s_e(2\pi)^2\rightarrow H^{s}_e(2\pi)^2$ is compact, i.e., the operator $(\widetilde{S}_0^\mathcal{J})^{-1}\widetilde{S}_\omega^{\rm{w}}$ is a compact perturbation of the identity operator. Then we conclude from the injectivity of $S_\omega$ together with Lemma~\ref{link} that the operator $(\widetilde{S}_0^\mathcal{J})^{-1}\widetilde{S}_\omega^{\rm{w}}$ is injective for $s>0$. Therefore, the bicontinuity of the operator $\widetilde{S}_\omega^{\rm{w}}: H^s_e(2\pi)^2\rightarrow H^{s+1}_e(2\pi)^2, s>0$ follows immediately from the Fredholm alternative. Let the operator $\widetilde{V}_\omega^{\rm{w}}$ be defined as
\ben
\widetilde{V}_\omega^{\rm{w}}[\gamma](\theta) = \int_{0}^\pi \left[ (\mu+\widetilde{\mu})^2A\Pi_\omega(x(\cos\theta),x(\cos\vartheta))A +2(\mu +\widetilde{\mu})G_{k_s}(x(\cos\theta),x(\cos\vartheta))\mathbb{I}\right] \gamma(\vartheta)d\vartheta.
\enn
In view of (\ref{dec2}), it can be proved that $\widetilde{V}_\omega^{\rm{w}}-\widetilde{V}_0: H^s_e(2\pi)^2\rightarrow H^{s+3}_e(2\pi)^2$ and $\widetilde{V}_\omega^{\rm{w}}: H^s_e(2\pi)^2\rightarrow H^{s+1}_e(2\pi)^2$ are continuous for all $s\ge 0$. Analogous to $\widetilde{S}_\omega^{\rm{w}}$, it can be proved that $\widetilde{V}_\omega^{\rm{w}}: H^s_e(2\pi)^2\rightarrow H^{s+1}_e(2\pi)^2, s>0$ is bicontinuous.

\begin{itemize}
\item Bicontinuity of $\widetilde{N}_\omega^{\rm{w}}$:
\end{itemize}

Extension of the regularized formula (\ref{N2}) to the open-surface case gives
\be
\label{Nce2}
\widetilde{N}_\omega^{\rm{w}}= \widetilde{N}_\omega^{\rm{g_1}} +\widetilde{N}_\omega^{\rm{g_2}} +\widetilde{N}_\omega^{\rm{pv}}
\en
where
\ben
\widetilde{N}_\omega^{\rm{g_1}}[\gamma](\theta)= \int_{0}^\pi \Pi_\omega^N(x(\cos\theta),x(\cos\vartheta)) \sin^2\vartheta \gamma(\vartheta)\mathcal{J}(\cos\vartheta)d\vartheta,
\enn
\ben
\widetilde{N}_\omega^{\rm{g_2}}[\gamma](\theta)= R_1\widetilde{T}_0[\gamma](\theta)+ \widetilde{Z}_0^{-1}\widetilde{D}_0R_2[\gamma](\theta),
\enn
and
\ben
\widetilde{N}_\omega^{\rm{pv}}[\gamma](\theta)= \widetilde{Z}_0^{-1}\widetilde{D}_0 \widetilde{V}_\omega^{\rm{w}}\widetilde{T}_0[\gamma](\theta).
\enn
Here,
\ben
\Pi_\omega^N(x,y)= -\rho\omega^2 (\nu_x\nu_y^\top-\nu_x^\top\nu_yI)G_{k_s}(x,y)+\widetilde{\mu}k_s^2G_{k_s}(x,y)J_{\nu_x,\nu_y}+ \rho\omega^2G_{k_p}(x,y)\nu_x\nu_y^\top,
\enn
and the operators $R_1,R_2$ are given by
\ben
R_1[\gamma](\theta)= -(\mu+\widetilde{\mu})\int_{0}^\pi \nu_{x(\cos\theta)}\nabla_{x(\cos\theta)}^\top [G_{k_s}(x(\cos\theta),x(\cos\vartheta))-G_{k_p}(x(\cos\theta),x(\cos\vartheta))]A \gamma(\vartheta)d\vartheta,
\enn
\ben
R_2[\gamma](\theta)= -(\mu+\widetilde{\mu})\int_{0}^\pi A\nabla_{x(\cos\vartheta)} [G_{k_s}(x(\cos\theta),x(\cos\vartheta))-G_{k_p}(x(\cos\theta),x(\cos\vartheta))]\nu_{x(\cos\vartheta)}^\top \gamma(\vartheta)d\vartheta.
\enn
The compactness of  $\widetilde{N}_\omega^{\rm{g_1}}: H^{s+1}_e(2\pi)^2\rightarrow H^{s}_e(2\pi)^2$ follows immediately due to its weakly-singular kernel and the compact embedding $H^{s+2}_e(2\pi)^2\hookrightarrow H^{s}_e(2\pi)^2$. Utilizing the singularity decomposition (\ref{dec3}), it can be proved that the operators $R_1,R_2: H^s_e(2\pi)^2\rightarrow H^{s+2}_e(2\pi)^2, s\ge 0$ are continuous. Note that $\widetilde{D}_0: H^{s+2}_e(2\pi)^2\rightarrow H^{s}_e(2\pi)^2$ and $\widetilde{T}_0: H^{s+1}_e(2\pi)^2\rightarrow H^{s}_e(2\pi)^2$ are continuous, it follows that $\widetilde{N}_\omega^{\rm{g_2}}: H^{s+1}_e(2\pi)^2\rightarrow H^{s}_e(2\pi)^2$ is compact. In addition, $\widetilde{N}_\omega^{\rm{pv}}: H^{s+1}_e(2\pi)^2\rightarrow H^{s}_e(2\pi)^2$ and $\widetilde{N}_\omega^{\rm{pv}}-\widetilde{N}_0^\mathcal{J}: H^{s+1}_e(2\pi)^2\rightarrow H^{s+1}_e(2\pi)^2$ are bounded in view of the expression
\ben
\widetilde{N}_\omega^{\rm{pv}}= \widetilde{N}_0^\mathcal{J}+ \widetilde{Z}_0^{-1}\widetilde{D}_0 (\widetilde{V}_\omega^{\rm{w}}-\widetilde{V}_0)\widetilde{T}_0,
\enn
and the boundedness of the operators $\widetilde{N}_0^\mathcal{J}: H^{s+1}_e(2\pi)^2\rightarrow H^{s}_e(2\pi)^2$, $\widetilde{D}_0: H^{s+2}_e(2\pi)^2\rightarrow H^{s}_e(2\pi)^2$, $\widetilde{T}_0: H^{s+1}_e(2\pi)^2\rightarrow H^{s}_e(2\pi)^2$ and $\widetilde{V}_\omega^{\rm{w}}-\widetilde{V}_0: H^s_e(2\pi)^2\rightarrow H^{s+3}_e(2\pi)^2$.

Since $\widetilde{N}_0^\mathcal{J}$ is bicontinuous, we rewrite $\widetilde{N}_\omega^{\rm{w}}$ as
\ben
\widetilde{N}_\omega^{\rm{w}}=\widetilde{N}_0^\mathcal{J}\left(I+ (\widetilde{N}_0^\mathcal{J})^{-1} (\widetilde{N}_\omega^{\rm{g_1}} +\widetilde{N}_\omega^{\rm{g_2}} +\widetilde{N}_\omega^{\rm{pv}}-\widetilde{N}_0^\mathcal{J})\right).
\enn
The compactness of $\widetilde{N}_\omega^{\rm{g_1}} +\widetilde{N}_\omega^{\rm{g_2}} +\widetilde{N}_\omega^{\rm{pv}}-\widetilde{N}_0^\mathcal{J}: H^{s+1}_e(2\pi)^2\rightarrow H^{s}_e(2\pi)^2$ and the boundedness of $(\widetilde{N}_0^\mathcal{J})^{-1}: H^{s}_e(2\pi)^2\rightarrow H^{s+1}_e(2\pi)^2$ imply  that $(\widetilde{N}_0^\mathcal{J})^{-1}\widetilde{N}_\omega^{\rm{w}}: H^{s+1}_e(2\pi)^2\rightarrow H^{s+1}_e(2\pi)^2$ is a compact perturbation of the identity operator. Therefore, bicontinuity of the operator $\widetilde{N}_\omega^{\rm{w}}: H^{s+1}_e(2\pi)^2\rightarrow H^s_e(2\pi)^2$ can be deduced by means of Fredholm alternative together with the injectivity of $N_\omega$ and Lemma~\ref{link}.

\begin{itemize}
\item Calder\'on relation:
\end{itemize}

It easily follows that
\ben
\widetilde{N}_\omega^{\rm{w}}\widetilde{S}_\omega^{\rm{w}}= \widetilde{J}_0^\mathcal{J}+ \widetilde{K},
\enn
where $\widetilde{N}_0^\mathcal{J}=\widetilde{N}_0^\mathcal{J}\widetilde{S}_0^\mathcal{J}$ is bicontinuous (see Corollary~\ref{corollaryJ}) and
\ben
\widetilde{K}= \widetilde{N}_\omega^{\rm{w}}(\widetilde{S}_\omega^{\rm{w}}-\widetilde{S}_0^\mathcal{J})+ (\widetilde{N}_\omega^{\rm{w}}-\widetilde{N}_0^\mathcal{J})\widetilde{S}_0^\mathcal{J}
\enn
maps compactly from $H^{s}_e(2\pi)^2$ into $H^{s}_e(2\pi)^2$. This completes the proof.
\end{proof}

\section{Conclusion}
\label{sec:5}

In this work, we have studied the spectral properties of the Calder\'on formulas associated with the elastic closed- and open-surface scattering problems in two dimensions. A generalized form of traction operator and the integral equation operators involving explicit edge singularities of potentials on open-surfaces are studied. It is proved that the Calder\'on formula is a compact perturbation of an identity operator and a bounded invertible operator, whose point spectrum is explicitly given, for the closed-surface case and open-surface case, respectively. The related Calder\'on formulas for three-dimensional elastic problems and the electromagnetic integral operators on open-surfaces (with more complex edge singularities) are left for future works.

\section*{Acknowledgments}
The work of LWX is supported by a Key Project of the Major Research Plan of NSFC (No. 91630205), and NSFC Grants (No.11771068, No.12071060). The work of TY is supported by an NSFC Grant (No. 12171465).

\section*{Appendix. Singularity decompositions}
\appendix
\renewcommand{\theequation}{A.\arabic{equation}}

The definition of the fundamental solution $\Pi_\omega$ gives
\ben
\Pi_\omega(x,y)&=&  \frac{i}{4\mu}H_0^{(1)}(k_s|x-y|)\mathbb{I} \nonumber\\
&\quad& -\frac{i}{4\rho\omega^2|x-y|} \left[k_sH_1^{(1)}(k_s|x-y|)-k_pH_1^{(1)}(k_p|x-y|)\right]\mathbb{I} \nonumber\\
&\quad& +\frac{i(x-y)(x-y)^\top}{4\rho\omega^2|x-y|^2} \left[k_s^2H_2^{(1)}(k_s|x-y|)-k_p^2H_2^{(1)}(k_p|x-y|)\right] ,\quad x\ne y.
\enn
From the series expansion of the bessel functions $J_n, Y_n$ (see~\cite[(10.2.2), (10.8.1)]{OLBC10}) and $H_n^{(1)}=J_n+iY_n$, it follows that
\ben
H_0^{(1)}(k|x-y|)= \sum_{m=0}^\infty \left[ \left(C_m^{(1)}+C_m^{(2)}\ln\frac{k}{2}\right) k^{2m}|x-y|^{2m} +C_m^{(2)}k^{2m}|x-y|^{2m}\ln|x-y|\right],
\enn
\ben
k_sH_1^{(1)}(k_s|x-y|)-k_pH_1^{(1)}(k_p|x-y|) = \sum_{m=0}^\infty \left( C_m^{(3)}|x-y|^{2m+1} +C_m^{(4)}|x-y|^{2m+1}\ln|x-y|\right),
\enn
and
\ben
&\quad&k_s^2H_2^{(1)}(k_s|x-y|)-k_p^2H_2^{(1)}(k_p|x-y|)\\
&=& \sum_{m=0}^\infty \left( C_m^{(5)}|x-y|^{2m+2} +C_m^{(6)}|x-y|^{2m+2}\ln|x-y|\right)-\frac{i(k_s^2-k_p^2)}{\pi},
\enn
where the constants $C_m^j, j=1,\cdots 6$ are given in~\cite[Lemma 3.1]{BXY17}. In particular,
\ben
C_0^{(2)}=\frac{2i}{\pi},\quad C_0^{(4)}=\frac{i(k_s^2-k_p^2)}{\pi}.
\enn
Note that
\ben
\frac{k_s^2-k_p^2}{4\pi\rho\omega^2}=\frac{1}{4\pi}\left(\frac{1}{\mu}-\frac{1}{\lambda+2\mu}\right)= \frac{\lambda+\mu}{4\pi\mu(\lambda+2\mu)}.
\enn
Then we conclude that
\be
\label{dec1}
\Pi_\omega(x(t),x(\iota))= \Pi_\omega^{(1)}(t,\iota)\ln|t-\iota|+C_{\lambda,\mu}^{(2)}\begin{pmatrix}
1 & 0 \\
0 & 0
\end{pmatrix}
+ \Pi_\omega^{(2)}(t,\iota),
\en
where $\Pi_\omega^{(1)}(t,\iota), \Pi_\omega^{(2)}(t,\iota)$ are smooth functions and particularly, $\Pi_\omega^{(1)}(t,\iota)$ can be expressed for all $m\in\N$ in the form
\ben
\Pi_\omega^{(1)}(t,\iota) = -C_{\lambda,\mu}^{(1)}\mathbb{I}+ \sum_{n=2}^{m+3} A_n(t)(t-\iota)^n+ \hat{A}_{m+3}(t,\iota)(t-\iota)^{m+4},
\enn
with $\hat{A}_{m+3}(t,\iota)$ being a smooth function.

Regarding to the hyper-singular integral operator $\widetilde{N}_\omega^{\rm w}$, we obtain analogously that
\be
\label{dec2}
&\quad&(\mu+\widetilde{\mu})^2A\Pi_\omega(x(t),x(\iota))A+2(\mu+\widetilde{\mu})G_{k_s}(x(t),x(\iota))\mathbb{I} \nonumber\\
&=& \frac{i[-(\mu+\widetilde{\mu})^2+2\mu(\mu+\widetilde{\mu})]}{4\mu}H_0^{(1)}(k_s|x-y|)\mathbb{I} \nonumber\\
&\quad& +\frac{i(\mu+\widetilde{\mu})^2}{4\rho\omega^2|x-y|} \left[k_sH_1^{(1)}(k_s|x-y|)-k_pH_1^{(1)}(k_p|x-y|)\right]\mathbb{I} \nonumber\\
&\quad& +\frac{i(\mu+\widetilde{\mu})^2(x-y)(x-y)^\top}{4\rho\omega^2|x-y|^2} \left[k_s^2H_2^{(1)}(k_s|x-y|)-k_p^2H_2^{(1)}(k_p|x-y|)\right] \nonumber\\
&\quad& -\frac{i(\mu+\widetilde{\mu})^2}{4\rho\omega^2} \left[k_s^2H_2^{(1)}(k_s|x-y|)-k_p^2H_2^{(1)}(k_p|x-y|)\right]\mathbb{I}
\nonumber\\
&=&  \Pi_\omega^{(3)}(t,\iota)\ln|t-\iota|+C_{\lambda,\mu,\widetilde{\mu}}^{(2)}\begin{pmatrix}
1 & 0 \\
0 & 0
\end{pmatrix}
+ \Pi_\omega^{(4)}(t,\iota),
\en
where $\Pi_\omega^{(3)}(t,\iota), \Pi_\omega^{(4)}(t,\iota)$ are smooth functions and particularly, $\Pi_\omega^{(3)}(t,\iota)$ can be expressed for all $m\in\N$ in the form
\ben
\Pi_\omega^{(3)}(t,\iota)= -C_{\lambda,\mu,\widetilde{\mu}}^{(1)}\mathbb{I} + \sum_{n=2}^{m+3} B_n(t)(t-\iota)^n+ \hat{B}_{m+3}(t,\iota)(t-\iota)^{m+4},
\enn
with $\hat{B}_{m+3}(t,\iota)$ being a smooth function. In addition,
\be
\label{dec3}
\nabla _x[G_{k_s}(x,y)-G_{k_p}(x,y)]&=& -\frac{i(x-y)[k_sH_1^{(1)}(k_s|x-y|)-k_pH_1^{(1)}(k_p|x-y|)]}{4|x-y|}\\
&=& \Pi_\omega^{(5)}(t,\iota)\ln|t-\iota|+ \Pi_\omega^{(6)}(t,\iota),
\en
where $\Pi_\omega^{(5)}(t,\iota), \Pi_\omega^{(6)}(t,\iota)$ are smooth functions and particularly, $\Pi_\omega^{(5)}(t,\iota)$ can be expressed for all $m\in\N$ in the form
\ben
\Pi_\omega^{(5)}(t,\iota)= \sum_{n=1}^{m+2} C_n(t)(t-\iota)^n+ \hat{C}_{m+2}(t,\iota)(t-\iota)^{m+3},
\enn
with $\hat{C}_{m+2}(t,\iota)$ being a smooth function.

\end{document}